\documentclass[a4paper,10pt]{article}
\usepackage{amssymb,amsmath,amsthm}
\usepackage[all,cmtip,line]{xy}
\usepackage[T1]{fontenc} % For Guillemots

\newtheorem{Theorem}{Theorem}[section]

\newtheorem{Proposition}[Theorem]{Proposition}

\theoremstyle{definition}

\newtheorem{Remark}[Theorem]{Remark}

\newcommand\ie{i.e.\ }

\newcommand \al{\alpha}
\newcommand\be{\beta}
\newcommand\ga{\gamma}
\newcommand\de{\delta}

\newcommand\ze{\zeta}
\newcommand\et{\eta}
\renewcommand\th{\theta}

\newcommand\ka{\kappa}
\newcommand\la{\lambda}

\newcommand\ta{\tau}
\newcommand\ph{\varphi}

\newcommand\ps{\psi}
\newcommand\om{\omega}

\newcommand\Th{\Theta}

\newcommand\Om{\Omega}

\newcommand\ham{\rm{ham}}

\newcommand\Diff{\text{\rm Diff}}

\newcommand\vol{\text{\rm vol}}
\newcommand\OGr{\text{\rm Gr}}
\newcommand\Gr{\text{\rm Gr}}

\newcommand\ex{\text{\rm ex}}

\newcommand\dR{\text{\rm dR}}

\newcommand{\oo}{\infty}

\newcommand\g{\mathfrak g}

\newcommand\M{\mathcal M}

\newcommand\X{\mathfrak X}

\newcommand\z{\mathfrak z}

\newcommand\R{\mathbb R}
\newcommand\Z{\mathbb Z}

\begin{document}

\title
{Central extensions of Lie algebras of \\symplectic and divergence free vector fields}
\date{}
\author{Bas Janssens and Cornelia Vizman}

\maketitle

\begin{abstract}
In this review paper, we present several results on central extensions of the Lie algebra of symplectic (Hamiltonian) vector fields, and compare them to similar results for the Lie algebra of (exact) divergence free vector fields.
In particular, we comment on universal central extensions and integrability to the group level.
\end{abstract}

%%%%%%%%%%%%%%
%%%%%%%%%%%%%%%%%%%%

\section{Perfect commutator ideals}
The Lie algebras of symplectic and divergence free vector fields both have perfect commutator ideals: the Lie algebras of \emph{Hamiltonian} and \emph{exact divergence free} vector 
fields, respectively.

\subsection{Symplectic vector fields}
Let $(M,\omega)$ be a compact\footnote{%
The compactness assumption mainly serves to simplify the comparison; 
most results at the Lie algebra level extend to noncompact manifolds.
},
connected, symplectic manifold of dimension $2n$. 
Then the Lie algebra of \emph{symplectic vector fields} is denoted by
\[\X(M,\om) := \{X \in \X(M) \,;\, L_{X}\omega = 0\}\,.\]
The \emph{Hamiltonian vector field} $X_{f}$ of
$f \in C^{\infty}(M)$ is
the unique vector field such that $i_{X_{f}} \omega = -df$.
Since the Lie bracket of $X,Y \in \X(M,\om)$
is Hamiltonian with $i_{[X,Y]}\omega = -d\omega(X,Y)$,
the Lie algebra of Hamiltonian vector fields
\[
\X_{\ham}(M,\om) := \{X_{f}\,;\, f \in C^{\infty}(M)\}
\]
is an ideal in $\X(M,\om)$.
The map $X \mapsto [i_{X}\omega]$ identifies the
quotient of $\X(M,\om)$ by $\X_{\ham}(M,\om)$ with the abelian Lie algebra $H_{\mathrm{dR}}^{1}(M)$,
\begin{equation}\label{sesham}
0 \rightarrow \X_{\ham}(M,\om) \rightarrow \X(M,\om) \rightarrow H_{\mathrm{dR}}^{1}(M) \rightarrow 0\,.
\end{equation}

\begin{Proposition}\label{hamp} {\rm (\cite{Calabi70})} %,ALM74} 
The 
Lie algebra $\X_{\ham}(M,\om)$ of Hamiltonian vector fields is perfect,
and coincides with the commutator Lie algebra of the Lie algebra
$\X(M,\om)$ of symplectic vector fields. 
\end{Proposition}

We equip the space $C^{\infty}(M)$ of $\R$-valued, smooth  functions on $M$ 
with the \emph{Poisson bracket}
$\{f,g\} = \omega(X_{f},X_{g})$.
The kernel of the Lie algebra homomorphism $f \mapsto X_{f}$, the space 
 of constant functions, is central in $C^{\infty}(M)$.
We thus  obtain a central extension
\begin{equation}\label{back}
0 \rightarrow \R \rightarrow C^{\infty}(M) \rightarrow \X_{\ham}(M,\om) \rightarrow 0\,,
\end{equation}
which is split for compact $M$.

\subsection{Divergence free vector fields}
Let $(P,\mu)$ be a compact, 
connected manifold of dimension $m$, equipped with a 
volume form $\mu$.
Then the Lie algebra of \emph{divergence free vector fields} is denoted by
\[\X(P,\mu) := \{X \in \X(P) \,;\, L_{X}\mu = 0\}\,.\]
The \emph{exact divergence free vector field} of
$\al \in \Om^{m-2}(P)$ is
the unique vector field $X_{\al}$ such that $d\al = i_{X_{\al}} \mu$.
Since the Lie bracket of $X,Y \in \X(P,\mu)$
is exact divergence free with $i_{[X,Y]}\mu = di_Xi_Y\mu$,
the Lie algebra of exact divergence free vector fields
\[
\X_{\ex}(P,\mu) := \{X_{\al}\,;\,\al \in \Om^{m-2}(P)\}
\]
is an ideal in $\X(P,\mu)$.
The map $X \mapsto [i_{X}\mu]$ identifies the
quotient of $\X(P,\mu)$ by $\X_{\ex}(P,\mu)$ with the abelian Lie algebra $H_{\mathrm{dR}}^{m-1}(P)$,
\begin{equation}
0 \rightarrow \X_{\ex}(P,\mu) \rightarrow \X(P,\mu) \rightarrow H_{\mathrm{dR}}^{m-1}(P) \rightarrow 0\,.
\end{equation}

\begin{Proposition}\label{volp} {\rm (\cite{Li74})} 
The Lie algebra $\X_{\ex}(P,\mu)$ of exact divergence free vector fields is perfect
and coincides with the commutator Lie algebra of the Lie algebra
$\X(P,\mu)$ of divergence free vector fields. 
\end{Proposition}

%%%%%%%%%%%%%%%%%%%

\section{Lie algebra cocycles from closed forms}
For the symplectic vector fields as well as for the
divergence free vector fields, one obtains Lie algebra 2-cocycles
from closed forms and also from singular cycles on the manifold.

\subsection{Divergence free vector fields}
Each closed 2-form $\et$ on $P$ 
defines a continuous 2-cocycle on the Lie algebra of divergence free vector fields,
called a {\it Lichnerowicz cocycle}, by (cf.\ \cite{Li74})
\begin{equation}\label{lich}
\la_\et(X,Y)=\int_P\et(X,Y)\mu\,.
\end{equation}
Cohomologous forms define cohomologous cocycles,
so we get a linear map
\begin{equation}\label{iso1}
[\et]\in H^2_{\dR}(P)\mapsto [\la_\et]\in H^2(\X(P,\mu),\R)
\end{equation}
of $H^2_{\dR}(P)$ into the second continuous Lie algebra cohomology 
of $\X(P,\mu)$.
In \cite[\S 10]{R95}, Roger outlines a proof of the following statement
(we were not able to find a detailed proof in the literature):
\begin{Theorem}\label{CEvolume}
The map \eqref{iso1} is an isomorphism between 
$H_{\dR}^2(P)$ and the continuous second Lie algebra cohomology 
$H^2(\X_{\ex}(P,\mu),\R)$.
\end{Theorem}

If $m$ is the dimension of $P$, then each 
singular $(m-2)$-cycle $C$ 
defines a continuous Lie algebra 2-cocycle on $\X(P,\mu)$ by
\begin{equation}\label{lichsing}
\la_C(X,Y)=\int_Ci_Yi_X\mu\,.
\end{equation}
Again, homologous $(m-2)$-cycles yield cohomologous cocycles, so we obtain a linear map
\begin{equation}
[C] \in H_{m-2}(P,\R) \mapsto [\lambda_{C}] \in H^2(\X(P,\mu),\R)\,.
\end{equation}  

\begin{Proposition}{\rm (\cite{Vi10})}
If the cohomology classes $[\et]\in H^2_{\dR}(P)$ and $[C]\in H_{m-2}(P,\R)$ 
are Poincar\'e dual to each other,
then the cocycles $\la_\et$ and $\la_C$ lie in the same Lie algebra cohomology class in 
$H^2(\X(P,\mu),\R)$.
\end{Proposition}

%%%%%%%%%%%%%%%

\subsection{Symplectic vector fields}
Each closed 1-form $\al$ on the compact symplectic manifold $(M,\om)$
of dimension $2n$ determines a continuous Lie algebra 2-cocycle
on the Lie algebra of Hamiltonian vector fields,
the {\it Roger cocycle}, given by \cite{Ki90,R95}
\begin{equation}\label{rog}
\psi_\al(X_f,X_g):=\int_M f\al(X_g)\om^n/n!\,.
\end{equation}
This expression does not depend on the choice of the Hamiltonian function $f$.
More generally, since $C^{\infty}(M)$ is a 1-dimensional central extension of 
the perfect Lie algebra $\X_{\ham}(M,\om)$, 2-cocycles on $\X_{\ham}(M,\om)$ correspond with 2-cocycles on 
$C^{\infty}(M)$, and we will not distinguish between the two.

Cohomologous 1-forms define cohomologous 2-cocycles,
so we get a linear map
\begin{equation}\label{iso2}
[\al]\in H^1_{\dR}(M)\mapsto [\ps_\al]\in H^2(\X_{\ham}(M,\om),\R).
\end{equation}
The following statement appears first in \cite{R95}, and is proven in
 \cite[Thm.~4.17]{JV15}. In the next section we sketch an outline of this proof.
\begin{Theorem}%{\rm (\cite{JV15})}
The map \eqref{iso2} is an isomorphism between 
$H_{\dR}^1(M)$ and the continuous second Lie algebra cohomology 
$H^2(\X_{\ham}(M,\om),\R)$.
\end{Theorem}

Each singular $(2n-1)$-cycle $C\in Z_{2n-1}(M)$ defines a continuous 2-cocycle on $\X_{\ham}(M,\om)$ by
\begin{equation}\label{rogsing}
\psi_C(X_f,X_g):= \int_C fdg \wedge\omega^{n-1}/(n-1)!\,.
\end{equation}

\begin{Proposition}{\rm (\cite{JV15})}\label{PDforsymp}
If the cohomology class $[\al]\in H_{\dR}^1(M)$ is Poincar\'e dual to $[C]\in H_{2n-1}(M,\R)$,
then the cocycle $\ps_C$ and the Roger cocycle
$\ps_\al$ are cohomologous cocycles on $\X_{\ham}(M,\om)$.
\end{Proposition}

In contrast to the Lichnerowicz cocycles $\la_\et$ that are defined directly on 
$\X(P,\mu)$, not every Roger cocycle $\ps_\al$ can be extended from $\X_{\ham}(M,\om)$
to $\X(M,\om)$: the obstruction depends on the cohomology class 
$[\al]\in H_{\dR}^1(M)$.
To write it down, we introduce a skew symmetric 2-form and an alternating 4-form on $H^1_{\mathrm{dR}}(M)$ defined by
\begin{gather*}
([\alpha], [\beta]) := \int_M \alpha \wedge \beta \wedge \omega^{n-1}/(n-1)!\\
([\al],[\beta_1],[\beta_2],[\beta_3])=\int_M \al \wedge\beta_1\wedge\beta_2\wedge\beta_3 \wedge\om^{n-2}/(n-2)!\,.
\end{gather*}

\begin{Theorem}{\rm (\cite{Vi05})}
The Roger cocycle $\ps_\al$ on  $\X_{\ham}(M,\om)$ is extendible to a cocycle
on $\X(M,\om)$ if and only if the class $[\al]\in H^1_{\mathrm{dR}}(M)$ 
satisfies
\[
([\al],b_1,b_2,b_3) = 
\frac{1}{\mathrm{vol}_\om(M)}
\sum_{\mathrm{cycl}}([\al],b_1)(b_2,b_3)
\]
for all $b_1, b_2, b_3\in H^1_{\mathrm{dR}}(M)$,
where the symplectic volume is $\mathrm{vol}_\om(M)=\int_M\om^n/n!$.
\end{Theorem}
The proof uses the Hochschild-Serre spectral sequence associated to \eqref{sesham}.
For the $2n$-torus and for Thurston's symplectic manifold, all Roger cocycles extend to the Lie algebra of symplectic vector fields, while for a surface of genus $g\ge 2$ no Roger cocycle can be extended.

%%%%%%%

\section{Universal central extension}

A continuous, linearly split central extension
$\mathfrak{z} \rightarrow 
\widehat{\g} \rightarrow \g$
of the locally convex Lie algebra $\g$ is called \emph{universal} 
if for every continuous, linearly split central extension
$\mathfrak{a} \rightarrow \g^{\sharp}\rightarrow \g$, there exists a unique continuous Lie algebra homomorphism 
$\phi \colon \widehat{\g} \rightarrow \g^{\sharp}$ such that
the following diagram is commutative:
\begin{center}
\mbox{
\xymatrix
{
\mathfrak{z} \ar[r]&  
\widehat{\g}\ar[r]^q&
\,\g\,
\\
\mathfrak{a} \ar[r]\ar@{<-}[u]_{\phi|_{\mathfrak{z}}}&  
\g^{\sharp}\ar[r]\ar@{<-}[u]_{\phi} &
\,\g.%\ar[r]^-{\dd}
\ar@{<-}[u]_{\mathrm{Id}}
}
}
\end{center}
If a universal central extension exists, then it is 
unique up to continuous isomorphisms. 
If $\mathfrak{\z}$ is finite dimensional, then the continuous second Lie algebra cohomology satisfies
$H^2(\g,\R) \simeq \mathfrak{z}^*$.

\subsection{Divergence free vector fields.}

The Lie algebra $\X_{\ex}(P,\mu)$ of exact divergence free vector fields admits  
a central extension
\[
0 \rightarrow Z^{m-2}(P) \rightarrow \Om^{m-2}(P) \stackrel{q}{\rightarrow} \X_{\ex}(P,\mu) \rightarrow 0\,
\]
on the level of \emph{Leibniz algebras}
(this observation is due to Brylinski, cf.\ \cite{R95}).  
The projection $q$ is given by $q(\alpha) = X_{\alpha}$, and
the (left) Leibniz bracket on $\Om^{m-2}(P)$ is given by
\begin{equation}\label{leib}
[\al,\be]=L_{X_\al}\be=i_{X_\al}i_{X_\be}\mu+di_{X_\al}\be.
\end{equation}
Indeed, the left Leibniz identity $[\al,[\be,\ga]]=[[\al,\be],\ga]+[\be,[\al,\ga]]$
follows from $[X_\al,X_\be]=X_{[\al,\be]}$.
Since the Leibniz bracket is skew-symmetric modulo exact terms, 
the Fr\'echet space $\Om^{m-2}(P)/d\Om^{m-3}(P)$
inherits the Lie bracket
\[
[[\al],[\be]]=[i_{X_\al}i_{X_\be}\mu]\,.
\]
The map $q:[\al]\mapsto X_\al$ then becomes
a Lie algebra homomorphism, making 
\begin{equation}\label{univol}
 H^{m-2}_{\mathrm{dR}}(P) \to \Om^{m-2}(P)/d\Om^{m-3}(P)\stackrel{q}{\to} \X_{\ex}(P,\mu) 
\end{equation} 
a central extension of Lie algebras,
called the {\it Lichnerowicz central extension}.
A consequence of Theorem \ref{CEvolume} is:

\begin{Theorem}
The universal central extension of $\X_{\ex}(P,\mu)$ is
the Lichnerowicz central extension \eqref{univol}.
\end{Theorem}

\subsection{Symplectic vector fields.}
In order to describe the universal central extension of $\X_{\ham}(M,\omega)$,
we introduce the \emph{canonical homology} $H^{\mathrm{can}}_{\bullet}(M,\pi)$ of 
a Poisson manifold $(M,\pi)$. It
is defined as the homology of the complex $\Omega^{\bullet}(M)$, equipped with the 
degree decreasing differential (\cite{K84})
\[
\delta = i_{\pi}d - di_{\pi}\,.
\]
For the special case of a Poisson structure $\pi = \omega^{\sharp}$ 
derived from a symplectic form, 
the following result of Brylinski expresses $\delta$
in terms of the \emph{symplectic Hodge star operator}
$* \colon \Omega^{k}(M)\rightarrow \Omega^{2n-k}(M)$.
This is the natural analogue 
of the 
ordinary (Riemannian) Hodge star operator in the symplectic context,
defined uniquely by the requirement that
$
\alpha \wedge *\beta = \wedge^{k}\omega(\alpha,\beta) \omega^n/n!
$
for all $\alpha, \beta \in \Omega^{k}(M)$.
\begin{Theorem}{\rm (\cite{Br88})} If $\pi = \omega^{\sharp}$ for a symplectic manifold $(M,\omega)$,
then the differential $\de:\Om^k(M)\to\Om^{k-1}(M)$ is
\[\delta = (-1)^{k+1}*d*\,.\]
\end{Theorem}
In particular, this shows that the $*$-operator yields an isomorphism between the canonical complex 
\begin{equation}\label{cancplx}
\cdots
\longrightarrow
\Omega^{2}(M)
\stackrel{\delta}{\longrightarrow}
\Omega^{1}(M)
\stackrel{\delta}{\longrightarrow}
\Omega^{0}(M)
\rightarrow
0
\end{equation}
and the de Rham complex
\begin{equation}
\cdots
\longrightarrow
\Omega^{2n-2}(M)
\stackrel{d}{\longrightarrow}
\Omega^{2n-1}(M)
\stackrel{d}{\longrightarrow}
\Omega^{2n}(M)
\rightarrow
0\,.
\end{equation}
We thus obtain an isomorphism $H^{\mathrm{can}}_{k}(M,\omega^{\sharp}) \simeq H^{2n-k}_{\mathrm{dR}}(M)$.
\subsubsection{Description of the universal central extension}

The universal central extension of $\X_{\ham}(M,\omega)$ is now obtained from \eqref{cancplx} 
by quotienting out $\delta\Omega^{2}(M)$.
The  space $\Omega^1(M)/\delta \Omega^2(M)$  becomes a Lie algebra with the bracket (\cite{JV15})
\begin{equation}\label{dde}
[[\al],[\be]]=[\de\al\, d\de\be]\,,
\end{equation}
and
the Lie algebra homomorphism 
$[\al] \mapsto X_{\de\al}$ yields the central extension 
\begin{equation}\label{unihambis}
H^{\mathrm{can}}_{1}(M,\omega^{\sharp}) \to \Omega^1(M)/\delta\Omega^2(M)\stackrel{q}{\to}\X_{\ham}(M,\om)\,.
\end{equation}

A different description of this extension is obtained if we identify the Lie algebra  $\Omega^1(M)/\delta \Omega^2(M)$  with $\Omega^{2n-1}(M)/d\Omega^{2n-2}(M)$
using the symplectic Hodge star operator.
The Lie bracket \eqref{dde} becomes
\begin{equation}\label{libra}
[[\ga_1],[\ga_2]]=[f_1 df_2\wedge\om^{n-1}/(n-1)!],
\end{equation}
where the function $f \in C^\oo(M)$ is uniquely defined by the form
$\ga\in\Om^{2n-1}(M)$ via $f\om^n/n!=d\ga$.
The central extension \eqref{unihambis} thus becomes
\begin{equation}\label{uniham}
H^{2n-1}_{\dR}(M)
\to\Om^{2n-1}(M)/d\Om^{2n-2}(M)\to\X_{\ham}(M,\om)\,.
\end{equation}

\begin{Theorem}{\rm (\cite{JV15})}
The central extension \eqref{unihambis} of $\X_{\ham}(M,\om)$, and hence also 
the isomorphic extension \eqref{uniham},
is universal.
\end{Theorem}
\begin{proof}
We give only an outline of the proof, and refer to \cite{JV15} for details.
The main point is to classify continuous $\R$-values 2-cocycles on $C^{\infty}(M)$.
Every such cocycle  
gives rise to an antisymmetric derivation $D:C^\oo(M)\to C^\oo(M)'$ by $D(f)g=\ps(f,g)$.
One first shows that $\psi$ must be diagonal, in the sense that
$\ps(f,g)=0$ whenever the supports of $f$ and $g$ are disjoint. This implies that 
$D$ is support decreasing, hence a differential operator of locally finite order by Peetre's theorem.
Using the fact that $D$ is a derivation, an explicit calculation in local coordinates 
shows that this differential operator is of order 1. In fact, $\psi$ is of the form $\psi(f,g) = S(fdg)$ for some 
current $S \in \Omega^{1}(M)'$. The closedness of $\psi$ then implies that $S$ vanishes on 
$\delta \Omega^2(M)$, so that it factors through a functional $\bar S:\Omega^{1}(M)/\delta \Omega^2(M) \rightarrow \R$.

We get that every 2-cocycle $\ps$ on $\X_{\ham}(M,\om)$
is of the form $%\ps(q[\al],q[\be])=
\ps(X_{\de\al},X_{\de\be})=\bar S([\de\al\, d\de\be])$,
so $q^*\ps$ is the coboundary of $\bar S$, by \eqref{dde}.
Thus $q^* H^2(\X_{\ham}(M,\om),\R) =0$, which shows the universality
of the central extension \eqref{unihambis}.
\end{proof}

Since $H^2(\g,\R)$ is isomorphic to the dual of the centre of \eqref{unihambis}, 
we find in particular 
\begin{equation}
H^2(\X_{\ham}(M,\om),\R) \simeq H_{1}^{\mathrm{can}}(M,\omega^{\sharp})^*\,.
\end{equation}
Identifying $H_{1}^{\mathrm{can}}(M,\omega^{\sharp})^*$ with the singular homology $H_{2n-1}(M,\R)$, 
we find that the map $[C] \mapsto [\psi_{C}]$ of \eqref{rogsing} is an isomorphism 
\begin{equation}
H_{2n-1}(M,\R) \simeq H^2(\X_{\ham}(M,\om),\R)\,.
\end{equation}
Similarly, we can
identify $H_{1}^{\mathrm{can}}(M,\omega^{\sharp})^*$
with $H^1_{\mathrm{dR}}(M)$ using Poincar\'e duality. This shows that
the map $[\alpha] \mapsto [\psi_{\alpha}]$
of \eqref{rog} is an isomorphism
\begin{equation}
H^{1}_{\mathrm{dR}}(M) \simeq H^2(\X_{\ham}(M,\om),\R)\,,
\end{equation}
as conjectured by Roger in \cite{R95}.

\section{Integrability of Lie algebra extensions}

Not every Lie algebra 2-cocycle is integrable to a central Lie group extension: the main obstruction is the period homomorphism \cite{N02}.
In this section we give conditions for integrability of the restriction
of Lichnerowicz cocycles to $\X_{\ex}(P,\mu)$ as well as integrability
 of Roger cocycles both on $C^\oo(M)$ and $\X_{\ham}(M,\om)$.

\subsection{Non-linear Grassmannians.}
The {\it non-linear Grassmannian} $\Gr_{k}(M)$ is the space of $k$-dimensional, compact, oriented, embedded, 
submanifolds $N \subseteq M$.  It is a Fr\'echet manifold in a natural way, see 
\cite[\S 44]{KM97}.
The tangent space of $\OGr_k(M)$ at $N$ can
be naturally identified with the space of smooth sections of the normal
bundle $TN^\perp:=(TM|_N)/TN$.

We 
have a transgression map $\tau \colon \Omega^{k+r}(M) \rightarrow \Omega^{r}(\Gr_{k}(M))$ defined by 
$$
(\tau{\alpha})([Y_1],\ldots,[Y_r]) := \int_{N}i_{Y_r}\ldots i_{Y_{1}}(\alpha|_{N}).
$$
Here all $[Y_j]$ are tangent vectors at $N\in\OGr_k(M)$, \ie sections
of the normal bundle $TN^\perp$. The expression above is independent of the vector fields $Y_j$ on $M$ along $N$ chosen to represent $[Y_j]$.

The Lie group $\mathrm{Diff}(M)$ acts smoothly on $\Gr_{k}(M)$ by 
$(\ph , N) \mapsto \ph(N)$. For this action, 
the fundamental vector field $\tau_{X} \in \X(\Gr_{k}(M))$ of $X\in \X(M) $ is
given by  $\ta_X(N) = [X|_{N}]$.
The transgression enjoys the following functorial properties:
\begin{align}\label{fun}
&d\circ \tau = \tau \circ d,\quad\quad\quad \ph^* \circ \tau = \tau \circ \ph^*, \\
&i_{\ta_{X}}\circ \tau = \tau\circ i_{X},\quad L_{\ta_X} \circ \tau = \tau\circ L_{X}.\nonumber
\end{align}
In particular, a closed differential form $\al\in\Om^{k+2}(M)$ transgresses to a closed
2-form $\Omega = \tau \alpha$ on $\Gr_{k}(M)$.
Just like in the finite dimensional case, a (possibly degenerate) symplectic manifold 
$(\mathcal M,\Om)$ is called \emph{prequantizable} if it admits a 
\emph{prequantum bundle} $\mathcal P\to\mathcal M$, 
\ie a principal $S^1$-bundle with connection 
form $\Th\in\Om^1(\mathcal P)$ whose curvature is $\Om$.
\begin{Theorem}{\rm (\cite{I96} \cite%[Theorem 1]
{HV04})}\label{taual} 
Let $\al\in\Om^{k+2}(M)$ be a closed differential form with integral cohomology class. Then
the non-linear Grassmannian $\OGr_{k}(M)$ endowed with the closed 2-form $\Om=\tau\al$ 
is prequantizable.
\end{Theorem}

A \emph{quantomorphism} of a prequantum bundle
$\mathcal P$ is a connection-preserving automorphism.
We denote the quantomorphism group of $\mathcal P$
by
\[
\Diff(\mathcal P,\Th)=\{\ph\in\Diff(\mathcal P):\ph^*\Th=\Th\},
\]
and its identity component by $\Diff (\mathcal P,\Th)_0$.
We denote
by
$\Diff_{\ham}(\mathcal M,\Om)$ the group of Hamiltonian diffeomorphisms of
the underlying symplectic manifold
$({\mathcal M}, \Omega)$.
The {\it prequantization central extension} of $(\mathcal M,\Om)$
is (\cite{Ko70,S70})
\begin{equation}\label{kostant}
S^1\to\Diff(\mathcal P,\Th)_0\to\Diff_{\ham}(\mathcal M,\Om)\,.
\end{equation}
If $\mathcal M$ is finite dimensional, then this is a central extension of Fr\'echet Lie groups \cite{RS}. If $\mathcal{M}$ is infinite dimensional, then \eqref{kostant} is merely a central extension of topological groups. Nonetheless, we have the following theorem:

\begin{Theorem}{\rm (\cite{NV03})}\label{pullb}
Let $\mathcal{P} \rightarrow \mathcal{M}$ be a prequantum bundle over a connected, symplectic, 
Fr\'echet manifold $(\mathcal M,\Om)$, 
and let $G\to\Diff_{\ham}(\M,\Om)$ be a Hamiltonian action on $\mathcal{M}$ by a connected,
Fr\'echet Lie group $G$.  
Then
the pullback of the prequantization central extension \eqref{kostant} by a Hamiltonian Lie group action 
$G\to\Diff_{\ham}(\M,\Om)$ 
provides a  central Lie group extension of $G$. 
\end{Theorem}

The corresponding Lie algebra extension of $\g= \mathrm{Lie}(G)$ is defined by the Lie algebra 
2-cocycle $(X,Y)\mapsto \Om(X_{\M},Y_{\M})(x_0)$,
where $x_0\in\M$ and $X_{\M},Y_{\M}$ are the fundamental vector fields on $\M$ 
with generators $X,Y\in\g$.

\subsection{Divergence free vector fields}

The group $\Diff_{\ex}(P,\mu)$ of exact volume preserving diffeomorphisms, defined as the kernel of Thurston's flux homomorphism \cite[\S 3]{B},
is a Fr\'echet Lie group with Lie algebra $\X_{\ex}(P,\mu)$.
We show that under certain integrality conditions, 
the restriction of the Lichnerowicz cocycles to the commutator ideal 
$\X_{\ex}(P,\mu) \subseteq \X(P,\mu)$
are integrable to Lie group extensions of $\Diff_{\ex}(P,\mu)$.

Since the cohomology class $\frac{k}{\mathrm{vol}_\mu(P)}[\mu]$
is integral for every $k\in \Z$, the 
non-linear Grassmannian $\OGr_{m-2}(P)$ with symplectic form 
$\Om=\frac{k}{\mathrm{vol}_\mu(P)}\tau\mu$ is prequantizable
by Theorem \ref{taual}.
Since the natural action of $\Diff_{\ex}(P,\mu)$
on $\OGr_{m-2}(P)$ is Hamiltonian,
we can apply Theorem \ref{pullb}   
to the connected component $\M$ of $Q\in\OGr_{m-2}(P)$.
This yields the central Lie group extension 
\begin{equation}\label{extdiffexgroup}
S^1 \rightarrow \widehat{\Diff}_{\ex}(P,\mu) \rightarrow \Diff_{\ex}(P,\mu)\,.
\end{equation}
The corresponding Lie algebra extension of $\X_{\ex}(P,\mu)$  is described by a multiple of the 
cocycle  $\la_Q$ of equation \eqref{lichsing}, namely
\begin{equation}\label{lic}
\left(\frac{k}{\mathrm{vol}_\mu(P)}\tau \mu\right)_{Q}(\ta_X,\ta_Y)
= \frac{k}{\mathrm{vol}_\mu(P)}\int_{Q}i_Yi_X\mu.
\end{equation}

\begin{Theorem}\label{lint}{\rm (\cite{I96,HV04})}
Let  $Q\subset P$ be a codimension two embedded submanifold. Then the
Lie algebra extensions defined by integral multiples of the
cocycle $\frac{1}{\mathrm{vol}_\mu(P)}\la_{Q}$
integrate to central Lie group extensions of the group of exact volume preserving diffeomorphisms $\Diff_{\ex}(P,\mu)$.
\end{Theorem}

A class $[\eta]\in H^2_{\dR}(P)$ is integral if and only if  it is  
Poincar\'e dual to the singular homology class $[Q] \in  H_{m-2}(P,\R)$ of 
a closed, codimension 2 submanifold of $P$ (cf.\ \cite[\S 12]{BT82}). 
We infer that if $[\eta]$ is integral, then the integral multiples 
$\frac{k}{\mathrm{vol}_\mu(P)}\la_\et$ of the Lichnerowicz cocycle give rise to integrable Lie algebra extensions
of $\X_{\ex}(P,\mu)$. 

\subsection{Symplectic vector fields}
%\paragraph{Roger cocycle.}
We study the integrability of Roger cocycles in the special case 
of a compact, prequantizable, symplectic manifold $(M,\om)$ of dimension $2n$
with prequantum bundle $\pi \colon P \rightarrow M$.
The connection form $\th \in \Omega^1(P)$ is then a contact form on the $(2n+1)$-dimensional 
manifold $P$,
because 
\begin{equation}\label{mu}
\mu={\textstyle \frac{1}{(n+1)!}}\th\wedge (d\th)^n
\end{equation}
is a volume form. The contact manifold
$(P,\th)$ is called the \emph{regular} or \emph{Boothby-Wang} contact manifold associated to 
$(M,\omega)$ \cite{BW}.

The group of quantomorphisms $\Diff(P,\th)$ is a Fr\'echet Lie group. Its  
Lie algebra $\X(P,\th)=\{X\in\X(P):L_X\th=0\}$ can be identified with the 
Poisson Lie algebra $(C^{\infty}(M),\{\,\cdot\,,\,\cdot\,\})$, by means of
the Lie algebra isomorphism 
\[
\X(P,\theta) \rightarrow C^{\infty}(M): \quad X \mapsto \theta(X).
\]
Denote its inverse $C^{\infty}(M) \rightarrow \X(P,\theta)$ by $f \mapsto \zeta_{f}$, uniquely determined by two conditions:
$i_{\ze_f}\th=\pi^*f$ and $i_{\ze_f}d\th=-\pi^*df$.
Note that
$\zeta_f\in\X(P,\th)$ is exact divergence free 
with respect to the volume form \eqref{mu}.

\begin{Proposition}\label{inc}
The group of quantomorphisms $\Diff(P,\th)$ is a subgroup of $\Diff(P,\mu)$,
while its identity component $\Diff(P,\th)_0$ is a subgroup of $\Diff_{\ex}(P,\mu)$.
\end{Proposition}

Let $Q\in\Gr_{2n-1}(P)$ be a codimension two submanifold of $P$. 
By Theorem~\ref{lint},  the Lichnerowicz cocycle
$\frac{k}{\mathrm{vol}_\mu(P)}\la_Q$ on $\X_{\ex}(P,\mu)$, with $k\in\Z$,
can be integrated to a central extension of $\Diff_{\ex}(P,\mu)$. 
Its pullback along the inclusion $\iota \colon \Diff(P,\theta)_{0} \hookrightarrow \Diff_{\ex}(P,\mu)$
yields a central Lie group extension 
\[
S^1 \rightarrow \widehat{\Diff}(P,\theta)_0 \rightarrow \Diff(P,\theta)_0\,.
\]
The corresponding Lie algebra extension is described by the pullback along 
$\iota_* \circ \ze\colon C^\oo(M)\to\X_{\ex}(P,\mu)$
of the Lichnerowicz cocycle $\frac{k}{\mathrm{vol}_\mu(P)}\la_Q$.
One checks (\cite[Prop.~4.8]{JV}) that 
the resulting cocycle is $-\frac{k}{\mathrm{vol}_{\mu}(P)}\psi_{C}$, with
$\psi_{C}$ the 2-cocycle on $C^{\infty}(M)$ of equation \eqref{rogsing},
associated to the singular $(2n-1)$-cycle $C=\pi_*Q$ 
on $M$.
Using the relation $\vol_\mu(P)=\frac{2\pi}{n+1} \vol_{\om}(M)$ 
with $\vol_{\omega}(M) = \int_{M}\omega^n/n!$ the Liouville volume, we get  
the following integrability theorem:

\begin{Theorem}{\rm (\cite{JV})}\label{thm}
Let $\pi:(P,\th) \rightarrow (M,\om)$ be a prequantum line bundle, and
$Q\subseteq P$ a codimension two embedded submanifold. Then the
Lie algebra extensions defined by integral multiples of the
cocycle $${\textstyle \frac{n+1}{2\pi \,\vol_{\om}(M)}}\ps_{\pi_*Q}$$
integrate to  Lie group extensions of the quantomorphism group $\Diff(M,\theta)_{0}$.
\end{Theorem}

In other words, the classes $[\psi_{C}]\in H^2(C^{\infty}(M),\R)$ corresponding to the lattice 
\[{\textstyle \frac{n+1}{2\pi\,\vol_{\omega}(M)}} \pi_*H_{n-1}(P,\Z) \subseteq H_{n-1}(M,\R)\]
give rise to integrable cocycles of $C^{\infty}(M)$.
Combining Theorem \ref{thm} with Poincar\'e duality (Proposition \ref{PDforsymp}), 
one sees that the same holds for Roger cocycles $[\psi_{\alpha}]$
with $[\alpha]$ in the lattice
\[
{\textstyle \frac{n+1}{2\pi\,\vol_{\omega}(M)}} \pi_{!}H^{2}_{\mathrm{dR}}(P)_{\Z} 
\subseteq H^{n-1}_{\mathrm{dR}}(M)\,,
\]
where $H_{\dR}^{2}(P)_{\Z}$ denotes the space of integral cohomology classes, 
and the map $\pi_{!} \colon H_{\dR}^{k}(P) \rightarrow H_{\dR}^{k-1}(M)$ is 
given by fibre integration.

\begin{Remark}
By Lie's Second Theorem for regular Lie groups
\cite[\S 40]{KM97}, the Lie algebra homomorphism
$\ka:  \X_{\ham}(M)\to C^{\infty}(M)$
that splits \eqref{back},
\[
\kappa(X_{f}) := f - {\textstyle \frac{1}{\mathrm{vol}_\om(M)}}\int_{M} f \omega^n/n!\,,
\]
integrates to a group homomorphism on the universal covering group 
of the group of Hamiltonian diffeomorphisms:
\[
K:\widetilde\Diff_{\ham}(M, \om) \rightarrow \Diff(P,\th)_0\,.
\]
The above extensions of the quantomorphism group  can be pulled back by  $K$ to obtain extensions of the universal covering group $\widetilde\Diff_{\ham}(M, \om)$ that integrate the cocycles $\frac{n+1}{2\pi\,\vol_\om(M)}\ps_\al$ and $\frac{n+1}{2\pi\,\vol_\om(M)}\ps_C$, 
subject to the integrality conditions 
$[C]\in \pi_*H_{2n-1}(P,\Z)$ and $[\al]\in\pi_{!} H_{\dR}^{2}(P)_{\Z}$. 
\end{Remark}

%%%%%%%%%%%%%%%%%%%%%%%%%%%%%%%%%%%%%%%%%%%%%%%%%%%%%%%

\end{document}